\documentclass[a4paper,reqno,12pt]{amsart}
\usepackage[dvips]{graphicx}
\usepackage{graphics}
\usepackage{amsthm, amssymb, amsmath, amsfonts}

\DeclareMathOperator*\id{id}%

\newtheorem{thm}{Theorem}[section]

\newtheorem{prop}{Proposition}[section]
\newtheorem{cor}{Corollary}[section]

\newtheorem{rem}{Remark}[section]

\newcommand{\To}{\longrightarrow}
\def\1{\^{\i}}
\def\2{\u{a}}
\def\3{\c{s}}
\def\4{\^{a}}
\def\5{\c{t}}

\def\la{\langle}
\def\ra{\rangle}

\begin{document}
\title{Applications of general variational inequalities to coincidence point results}
\author{Alireza Amini-Harandi$^\dag$}
        \address{Department of Pure Mathematics\newline University of Shahrekord \newline Shahrekord, 88186-34141, Iran\newline School of Mathematics\newline Institute for Research in Fundamental Sciences (IPM)\newline P.O. Box: 19395-5746, Tehran, Iran}
        \email{ aminih$\underline{~}$a@yahoo.com}
\thanks{$^\dag$ The author was partially supported by a grant from IPM (No. 92470412).}
        \author{Szil\'ard L\'aszl\'o$^*$}
             \address{ Department of Mathematics, \newline Technical University of Cluj-Napoca,\newline
Str. Memorandumului nr. 28, 400114 Cluj-Napoca, Romania.}
\email{laszlosziszi@yahoo.com}
             \thanks{$^*$This work was supported by a grant of the Romanian Ministry of Education, CNCS - UEFISCDI, project number PN-II-RU-PD-2012-3 -0166.}

\begin{abstract}
In this paper we obtain some  existence result of solution for general variational inequalities. As applications several coincidence and fixed point results are provided.
\newline


\end{abstract}

\maketitle
\section{Introduction}

The variational inequality theory has its origin in the works   of Stampacchia (see \cite{S}) and Fichera (see \cite{F}). This theory provides us very powerful techniques for studying problems arising in various  branches of mathematics but also in mechanics,  transportation, economics equilibrium or contact problems in elasticity. For instance, the moving boundary value problem can be characterized by a class of variational inequalities, the traffic assignment problem is a variational inequality problem, the free boundary value problem can be studied effectively in the framework of variational inequalities (see \cite{Bai, Ber-Gaf, Daf}).

During the last decades, the variational inequality theory has imposed itself, due to its importance, and has been  subject of a  dynamical evolution reflected in a number of generalizations,  studied and applied in various directions (see, for instance, \cite{BL,Fe,F,KS,L,N1,S}).

General variational inequalities, which were introduced and studied by Noor, can be used to study a wide class of problems including unilateral, moving boundary, obstacle, free boundary and equilibrium problems arising in various areas of pure and applied sciences (see \cite{N1}).

Let $H$ be a Hilbert space and let $C$ be a nonempty, closed and convex subset of $H$. Noor considered the following problem, called general variational inequality (see \cite{N1, Nor1,Noor}). For two continuous mappings $T,g:H\To H$, find $u\in g^{-1}(C)$ such that
$$\la T(u),g(v)-g(u)\ra \ge 0,\mbox{  for all }v\in g^{-1}(C).$$

Noor remarked (see \cite{N1}), that  if  $C^* =\{u\in H| \la u,v\ra\ge0,\,\mbox{\,for all\,}\, v\in C\}$ is a polar cone of the convex cone $C\subseteq H$,
then the general variational inequality is equivalent to finding $u\in H$, such that
$g(u)\in C$, $T(u)\in C^*$,  $\la T(u),g(u)\ra=0,$ which is known as the general nonlinear complementarity problem.

If $C=H$, then the general variational inequality problem is equivalent to finding $u\in H$, such that
$\la T(u),g(v)\ra = 0, \mbox{  for all }g(v)\in H$ which is known as the weak formulation of the boundary value problem.

For a concrete third-order obstacle boundary value problem, that may be characterized by a general variational inequality see \cite{Nor2,No}.

The problem of general variational inequalities has been extended by Noor to the case when the operators involved are set-valued.

Let $C(H)$ be the family of all nonempty compact subsets of $H.$ Let $T:H\To C(H)$ be a set-valued operator,
and let $g:H\To H$ be a single-valued operator. Let $K$ be a nonempty, closed and convex set in $H.$
Consider the problem of finding $x\in H,\, g(x)\in K,\,u\in T(x)$ such that
\begin{equation*}
\la u,g(y)-g(x)\ra \ge0,\mbox{  for all } g(y)\in K.
\end{equation*}
This problem is called a multivalued variational inequality. It has been shown, that a wide class of multivalued odd-order and nonsymmetric,
free, obstacle, moving equilibrium and optimization problems arising in pure and applied sciences can be studied via the multivalued variational
inequality, ( see \cite{Noor3}).

While existence results of the solution for the classical Stampacchia variational inequalities were abundant in the last years (see for instance \cite{MR}), this is not the case of general variational inequality, respectively of multivalued variational inequality. Some variants of  the general variational inequality problem, respectively the multivalued variational inequality problem (see \cite{Y}) have also been  studied in \cite{L,L1} and \cite{L2} in a Banach space context.  In these papers several existence results of the solution for these problems were established in the case the operator $g$ is of type ql and the operators involved possess some continuity properties.  Recall that the operator $g:D\subseteq X\To Y$ is of type ql (see \cite{L}), if for every $x,y\in D,$ and every $z\in [x,y]\cap D,$ one has $g(z)\in[g(x),g(y)],$ where  $X$ and $Y$ are real linear spaces and $[x,y]$ denotes the closed line segment with the endpoints $x$ respectively, $y.$  Moreover, it has been shown by examples that the existence results of solution for these problems, obtained in the papers mentioned above, fail outside of the class of ql type operators.

  In this paper we  obtain some existence results of the solution for general variational inequalities without assuming that the operators involved are of type ql. We do not assume any  continuity property of the operators involved, instead we work with some sequential conditions imposed on these operators. We use these results to obtain some new coincidence point results in Hilbert spaces.

 In what follows let $X$ be a real Banach space and let $X^*$ be the topological dual of $X.$ We denote by $\la x^*,x\ra$  the value of the linear and continuous functional $x^*\in X^*$ in $x\in X.$ Consider the set $K\subseteq X,$ and let  $A:K\To X^*$ and $a:K\To X$ be two given operators.

We deal with the following formulation of the general variational inequality problem, denoted by  $VI_S(A,a,K)$ (see \cite{L}). Find an element $x\in K,$ such that
$$\la A(x),a(y)-a(x)\ra\geq 0,\mbox{ for all }y\in K.$$

We denote by $\id_K$ the identity mapping on $K$, that is, $$\id\nolimits_{K}:K\To K,\, \id\nolimits_K(x)=x,\mbox{ for all } x\in K.$$

Obviously, when $a\equiv \id_K,$ then $VI_S(A,a,K)$  reduces to Stampacchia variational inequality, $VI_S(A,K)$ (see \cite{S}), that is, find $x\in K$ such that
$$\la A(x),y-x\ra\geq 0,\mbox{ for all }y\in K.$$

The outline of the paper is the following. In the next section  we obtain several existence result of the solution for general variational inequalities.  These results  will be used  in Section 3 for providing some unknown coincidence point results in Hilbert spaces.  Also here, as corollaries, several fixed point theorems are obtained.

\section{Solution existence}
In the sequel let $X$ be a real Banach space and let $X^*$ be the topological dual of $X.$  Recall that the operator $T:D\subseteq X\To X^*$ is called weak to $\|\cdot\|$-sequentially continuous at $x\in D$, if for every sequence $\big(x_n\big)\subseteq D$ that converges weakly to $x\in D,$  the sequence $\big(T(x_n)\big)\subseteq X^*$ converges to $T(x)\in X^*$ in the topology of the norm of $X^*.$ We say that $T$ is  weak to $\|\cdot\|$-sequentially continuous on  $D\subseteq X,$ if has this property at every point  $x\in D$. Let us mention, that the compact operators between two Banach spaces, in particular  the Fredholm integral operators, have this continuity property.

 The operator $T:D\subseteq X\To X$ is called weak to weak-sequentially continuous at $x\in D$, if for every sequence $\big(x_n\big)\subseteq D$ that converges weakly to $x\in D,$  the sequence $\big(T(x_n)\big)\subseteq X$ converges weakly to $T(x)\in X.$ We say that $T$ is  weak to weak-sequentially continuous on  $D\subseteq X,$ if has this property at every point  $x\in D$.

The following result was established in \cite{L}.

\begin{prop}\label{p2.1}(Corollary 4.1 \cite{L}) Let  $A:K\subseteq X\To X^*$  be a given operator. If $A$ is weak to $\|\cdot\|$-sequentially continuous and $K$ is weakly compact and convex, then Stampacchia variational inequality, $VI_S(A,K)$, admits solutions.
\end{prop}

Now we are able to provide a result concerning on existence of the solution of $VI_S(A,a,K)$. We need the following concept. Let $Z$ and $Y$ be two arbitrary sets. Recall that the inverse of a mapping $f:Z\To Y$ is defined as the set-valued mapping $f^{-1}:Y\rightrightarrows Z,\, f^{-1}(y)=\{z\in Z:f(z)=y\}.$  A single valued selection of the set-valued map $F:Z\rightrightarrows Y$ is the single valued map $f:Z\To Y$ satisfying $f(z)\in F(z)$ for all $z\in Z.$

\begin{thm}\label{t3.1} Let $K\subseteq X$ and let $A:K\To X^*$ and $a:K\To X$ be two given operators. Assume that $a(K)$ is weakly compact and  convex. Assume further, that for every sequence $\big(x_n\big)\subseteq K$ the following condition is satisfied:  if the sequence $\big(a(x_n)\big)\subseteq a(K)$ converges weakly to $a(x)\in a(K)$ then the sequence $\big(A(x_n)\big)\subseteq X^*$ is norm convergent to $A(x)\in X^*.$ Then $VI_{S}(A,a,K)$ admits solutions.
\end{thm}
\begin{proof} Consider  $b:a(K)\To K$  a single valued selection of $a^{-1}.$ Let $\big(u_n\big)\subseteq a(K)$ a weakly convergent sequence to $u\in X.$ Then due to the weak compactness of $a(K)$ we have $u\in a(K).$ We show that $(A\circ b)(u_n)\To (A\circ b)(u),\, n\To\infty.$
Since $\big(u_n\big)\subseteq a(K)$ there exists a sequence $\big(x_n\big)\subseteq K$ such that $u_n=a(x_n),\,n\in \mathbb{N} .$ Analogously $u=a(x)$ for some $x\in K.$ Note that $a(b(u_n))=u_n,\,n\in \mathbb{N} $ and $a(b(u))=u,$ hence the sequence $\big(a(b(u_n))\big)$ converges weakly to $a(b(u)).$
According to the hypothesis of the theorem  $$(A\circ b)(u_n)\To (A\circ b)(u),\, n\To\infty.$$

 Hence, the operator $A\circ b:a(K)\To X^*$ is weak to $\|\cdot\|$-sequentially continuous. According to Proposition \ref{p2.1} there exists $u\in a(K)$ such that,
 \[\la (A\circ b)(u),v-u\ra\geq 0,\mbox{ for all }v\in a(K).\]
 Since for every $y\in K$ there exists $v\in a(K)$ such that $a(y)=v$, we obtain the following general nonlinear variational inequality (see \cite{Noor4}),
 $$\la (A\circ b)(u), a(y)-u\ra \ge 0,\,\forall y\in K.$$
 Observe that $a(b(u))=u.$ Thus, \[\la A(b(u)), a(y)-a(b(u))\ra \ge 0,\mbox{ for all }y\in K,\] or equivalently, $b(u)\in K$ is a solution of $VI_S(A,a,K).$
\end{proof}

\begin{rem}\label{r3.1}\rm
The condition:  $\big(a(x_n)\big)\subseteq a(K)$ converges weakly to $a(x)\in a(K)$ then the sequence $\big(A(x_n)\big)\subseteq X^*$ is norm convergent to $A(x)\in X^*,$ in the hypothesis of Theorem \ref{t3.1} implies that $a^{-1}(a(x))\subseteq A^{-1}(A(x))$ for every $x\in K.$ Indeed, let $x\in K.$ Since $a(K)$ is weakly sequentially closed, there exists a sequence $\big(a(x_n)\big)\subseteq a(K)$ converging to $a(x)$ in the weak topology of $X$. But then the sequence $\big(A(x_n)\big)$ converges strongly to $A(x).$ Let $y\in a^{-1}(a(x)).$ Then $a(y)=a(x)$ hence $\big(A(x_n)\big)$ converges strongly to $A(y).$ Therefore $A(y)=A(x)$, hence $y\in A^{-1}(A(x)).$
\end{rem}

The next Corollary allows us to obtain the conclusion of Theorem \ref{t3.1}  in  conditions that can  more easily be verified.

\begin{cor}\label{c3.1} Assume that $K$ is weakly compact, $a$ is weak to weak-sequentially continuous and $a(K)$ is convex. Assume further,  that for every sequence $\big(x_n\big)\subseteq K$ the following condition holds:  if the sequence $\big(a(x_n)\big)\subseteq a(K)$ converges weakly to $a(x)\in a(K)$ then the sequence $\big(A(x_n)\big)\subseteq X^*$ is norm convergent to $A(x)\in X^*.$
Then $VI_S(A, a,K)$ admits solutions.
\end{cor}

\begin{proof} We show that $a(K)$ is weakly compact, and the conclusion follows from Theorem \ref{t3.1}.
By Eberlein-\u{S}mulian theorem, (see, for instance, \cite{FHHM}) $a(K)$ is weakly compact if and only if is weakly sequentially compact. To prove that $a(K)$ is weakly sequentially compact, let $\big(u_n\big)$ be an arbitrary sequence in $a(K)$. Then there exists a sequence $\big(x_n\big)\subseteq K$ such that $u_n=a(x_n),\,n\in \mathbb{N} .$ We show that $\big(a(x_n)\big)$ has a weakly convergent subsequence in $a(K)$. Since $\big(x_n\big)$ is a sequence in the weakly compact set $K$,  $\big(x_n\big)$ has a weakly convergent subsequence (note that by Eberlein-\u{S}mulian theorem $K$ is weakly sequentially compact, that is, every sequence in $K$ has a convergent subsequence). Let $\big(x_{n_i}\big)$ be a subsequence of $\big(x_n\big)$ that is weakly convergent to $x\in K$. Since $a$ is weak to weak sequentially continuous $\big(a(x_{n_i})\big)$  converges weakly to $a(x)$ and the proof is completed.
\end{proof}

 Recall that the operator $T:D\subseteq X\To {X^*}$ is called monotone (see\cite{B1,B2,M1,M2,Gia-Mau}) if for all $x,y\in D$ one has
$\la T(x)-T(y),x-y\ra\ge 0.$

  We say that $T$ is monotone relative to the operator $t:D\To X,$ if for all $x,y\in D,$ we have $\la T(x)-T(y),t(x)-t(y)\ra\ge 0$.

Obviously, if $t\equiv\id_D$ we obtain the definition of  monotonicity.  $T$ is called continuous on finite dimensional subspaces, if for every finite dimensional subspace $M\subseteq X$ the restriction of $T$ to ${D\cap M}$ is weak continuous, that is, for every sequence $\big(x_n\big)\subseteq D\cap M$ converging to $x\in M$ the sequence $\big(A(x_n)\big)\subseteq X^*$ converges to $A(x)$ in the weak topology of $X^*$ (see \cite{V}).

In what follows we present a well known existence results of solution for Stampacchia variational inequality $VI_S(A,K)$. The following classical result concerning on the existence of the solution of $VI(A,K),$ is due to Hartmann and Stampacchia (see \cite{HS}).

\begin{prop}\label{p2.2} Let $X$ be a reflexive Banach space, let $K$ be a weakly compact convex nonempty subset of $X$. If  $A:K\To X^*$ is a monotone operator, continuous on finite dimensional subspaces then $VI_S(A,K)$ admits solutions.
\end{prop}

In what follows we provide a result concerning on the existence of the solution of general variational inequalities.

\begin{thm}\label{t3.2} Assume that the Banach space $X$ is reflexive.  Let $A:K\subseteq X\To X^*$ be monotone relative to  $a:K\To X,$ where $a(K)$ is weakly compact and  convex. Assume further, that for every finite dimensional subset $L\subseteq a(K)$ and for every sequence $\big(x_n\big)\subseteq K,$ such that $a(x_n)\in L$ for every $n\in\mathbb{N},$ the following condition holds: if the sequence $\big(a(x_n)\big)\subseteq L$ converges to $a(x)\in a(K)$ then the sequence $\big(A(x_n)\big)\subseteq X^*$ is weakly convergent to $A(x)\in X^*.$  Then $VI_{S}(A,a,K)$ admits solutions.
\end{thm}
\begin{proof}
Consider  $b:a(K)\To K$  a single valued selection of $a^{-1}$ and let $u,v\in a(K).$ Then $\la (A\circ b)(u)-(A\circ b)(v),u-v\ra =\la A(x)-A(y),a(x)-a(y)\ra ,$ where $x=b(u),\,y=b(v).$ Since $A$ is monotone relative to $a$, we have $\la A(x)-A(y),a(x)-a(y)\ra \ge 0,$ hence, the operator  $A\circ b:a(K)\To X^*$ is monotone.

 Let  $M$ be a finite dimensional subspace of $X$ and let $L=M\cap a(K).$ Let   $\big(u_n\big)\subseteq L$ be a  sequence  convergent to $u\in a(K).$ Since $M$ is finite dimensional subspace it is closed. Hence, according to the weak compactness of $a(K)$ we get that   $u\in L.$ We have to show that the sequence $\big((A\circ b)(u_n)\big)\subseteq X^*$ converges to $(A\circ b)(u)\in X^*$ in the weak topology of $X^*.$
Since $\big(u_n\big)\subseteq a(K)$ there exists $\big(x_n\big)\subseteq K$ such that $u_n=a(x_n),\,n\in\mathbb{N} .$ Analogously $u=a(x)$ for some $x\in K.$  Since $b:a(K)\To K$ is a single valued selection of $a^{-1},$  observe that $a(b(u_n))=u_n\in L,\,n\in\mathbb{N} $ and $a(b(u))=u\in L.$ Hence $\big(a(b(u_n))\big)$ converges to $a(b(u)).$

According to the hypothesis of the theorem  the sequence $\big(A( b(u_n))\big)\subseteq X^*$ converges weakly to $A(b(u))\in X^*$ when $n\To\infty,$ which shows that $A\circ b$ is continuous on finite dimensional subspaces. Hence, according to Proposition \ref{p2.2} there exists $u\in a(K)$ such that,
 \[\la (A\circ b)(u),v-u\ra\geq 0,\mbox{ for all }v\in a(K).\]
 Since for every $y\in K$ there exists $v\in a(K)$ such that $a(y)=v$, we obtain $$\la (A\circ b)(u), a(y)-u\ra \ge 0,\,\forall y\in K.$$
 Observe that $a(b(u))=u.$ Thus, \[\la A(b(u)), a(y)-a(b(u))\ra \ge 0,\mbox{ for all }y\in K,\] or equivalently, $b(u)\in K$ is a solution of $VI_S(A,a,K).$
\end{proof}

\begin{rem}\label{r3.2}\rm Observe, that  if $K$ is weakly compact and $a$ is weak to weak-sequentially continuous then, according to the proof of Corollary \ref{c3.1}, $a(K)$ is weakly compact, hence we have the following corollary.
\end{rem}

\begin{cor}\label{c3.2}
Let the Banach space $X$ be reflexive.  Assume that $K$ is weakly compact, $a(K)$ is convex and $a$ is weak to weak-sequentially continuous. Let $A$ be monotone relative to  $a$ and assume further, that for every finite dimensional subset $L\subseteq a(K)$ and for every sequence $\big(x_n\big)\subseteq K,$ such that $a(x_n)\in L$ for every $n\in\mathbb{N},$ the following condition holds: if the sequence $\big(a(x_n)\big)\subseteq L$ converges to $a(x)\in a(K)$ then the sequence $\big(A(x_n)\big)\subseteq X^*$ is weakly convergent to $A(x)\in X^*.$  Then $VI_{S}(A,a,K)$ admits solutions.
\end{cor}

\section{Coincidence points}
In this section we obtain several coincidence point results for two given mappings by using the existence results of the solution for general variational inequalities established in the previous section. Let $X$ and $Y$ be two arbitrary sets and let $f,g \colon X \To Y$ be two given mappings. Recall that a point $x\in X$ is a coincidence point of $f$ and $g$ if $f(x) = g(x).$ If $X=Y$ and $g\equiv \id_X$ then a coincidence point $x\in X$ of $f$ and $g$ is called a fixed point of $f$, that is $f(x)=x.$ A considerable number of  problems concerning on  the existence of the solution of nonlinear inequalities, arising in different areas of mathematics, can be treated as a coincidence point or fixed point problem (see, for instance, \cite{NS-MAT}).

 In the sequel $H$ denotes a real Hilbert space identified with its dual.

The following coincidence point result is an easy consequence of Theorem \ref{t3.1}.

\begin{thm}\label{t4.1} Let $K\subseteq H$ and let $f,g:K\To H$ be two given mappings. Assume that $f(K)\subseteq g(K)$ and $g(K)$ is weakly compact and convex. Assume further, that for every sequence $\big(x_n\big)\subseteq K$ the following condition holds: if the sequence $\big(g(x_n)\big)\subseteq g(K)$ converges weakly to $g(x)\in g(K)$ then the sequence $\big(g(x_n)-f(x_n)\big)\subseteq H$ converges to $g(x)-f(x)\in H$ in the topology of the norm of $H$. Then $f$ and $g$ have a coincidence point.
\end{thm}
\begin{proof}  It can be easily verified that the operators $A:K\To H, \,A(x)=g(x)-f(x)$ and $a:K\To H,\,a(x)=g(x)$ satisfy the assumptions of the hypothesis of Theorem \ref{t3.1}. Hence, $VI_S(g-f,g,K)$ admits solutions.

Let  $x\in K$ be a solution of $VI_S(g-f,g,K)$. Then we have  $\la g(x)-f(x),g(y)-g(x)\ra \ge 0$ for all $y\in K.$ Since $f(K)\subseteq g(K)$ choose $y\in K$ such that $g(y)=f(x).$ Then $\la g(x)-f(x),f(x)-g(x)\ra \ge 0$, or equivalently $-\|f(x)-g(x)\|^2\ge 0$, which leads to $f(x)=g(x).$
\end{proof}

\begin{cor}\label{c4.1}
Let $K \subseteq H$ be a weakly compact set, let $f : K \To H$ and $g : K \To H$ be two mappings, such that $g$ is weak to weak-sequentially continuous and $g(K)$ is convex. Assume that for every sequence $\big(x_n\big)\subseteq K$ the following condition holds:  if the sequence $\big(g(x_n)\big)\subseteq g(K)$ converges weakly to $g(x)\in g(K)$ then the sequence  $\big(g(x_n)-f(x_n)\big)\subseteq H$ is norm convergent to $g(x)-f(x)\in H.$ If $f(K) \subseteq g(K)$ then  $f$ and $g$ have a coincidence point.
\end{cor}
\begin{proof} According to the proof of Corollary \ref{c3.1}, $g(K)$ is weakly compact. The conclusion follows from Theorem \ref{t4.1}.
\end{proof}

The next corollary provides sufficient conditions for the existence of a fixed point of a given mapping.

\begin{cor}\label{c4.2} Let $K \subseteq H$ be a weakly compact and convex set, let $f : K \To H$ be a given mapping such that $f(K) \subseteq K.$ Assume that for every sequence $\big(x_n\big)\subseteq K$ the following condition holds:  if the sequence $\big(x_n\big)\subseteq K$ converges weakly to $x\in K$ then the sequence  $\big(x_n-f(x_n)\big)\subseteq H$ is norm convergent to $x-f(x)\in H.$ Then $f$ has a fixed point.
\end{cor}
\begin{proof} The conclusion follows from Theorem \ref{t4.1} by taking $g\equiv \id_K.$
\end{proof}

In finite dimensional Hilbert spaces the weak topology and the topology of the norm are equivalent, hence we have the following corollaries.

\begin{cor}\label{c4.3} Let $K\subseteq \mathbb{R} ^n$ and let $f,g:K\To \mathbb{R} ^n$ be two given mappings. Assume that $f(K)\subseteq g(K)$ and $g(K)$ is  compact and convex. Assume further, that for every sequence $\big(x_n\big)\subseteq K$ the following condition holds:  if the sequence $\big(g(x_n)\big)\subseteq g(K)$ converges  to $g(x)\in g(K)$ then the sequence $\big(f(x_n)\big)\subseteq \mathbb{R} ^n$ converges to $f(x)\in\mathbb{R}^n.$ Then $f$ and $g$ have a coincidence point.
\end{cor}
 The next result can be viewed as  a coincidence point version of Brouwer fixed point theorem.

\begin{cor}\label{c4.4}
Let $K \subseteq \mathbb{R}^n$ be a compact set, let $f ,g: K \To \mathbb{R}^n$ be two given mappings, such that $g$ is continuous, $g(K)$ is convex and  $f(K) \subseteq g(K)$. Assume that for every sequence $\big(x_n\big)\subseteq K$ the following condition holds:  if the sequence $\big(g(x_n)\big)$ converges to $g(x)\in g(K)$ then the sequence $\big(f(x_n)\big)$ converges to $f(x)\in\mathbb{R}^n.$
  Then there exists $x_0 \in K$ such that $f(x_0) = g(x_0)$.
\end{cor}

The following coincidence point result is obtained via Theorem \ref{t3.2}.

\begin{thm}\label{t4.2} Let $K\subseteq H$ and let $f,g:K\To H$ be two given mappings, such that
$$\|g(x)-g(y)\|^2\ge \la f(x)-f(y),g(x)-g(y)\ra \mbox{ for all }x,y\in K.$$  Assume that $f(K)\subseteq g(K)$ and $g(K)$ is weakly compact and convex. Assume further, that for every finite dimensional subset $L\subseteq g(K)$ and for every sequence $\big(x_n\big)\subseteq K,$ such that $g(x_n)\in L$ for every $n\in\mathbb{N},$ the following condition holds: if the sequence $\big(g(x_n)\big)\subseteq L$ converges to $g(x)\in g(K)$ then the sequence $\big(f(x_n)\big)\subseteq H$ converges to $f(x)\in  H$ in the weak topology of  $H$. Then $f$ and $g$ have a coincidence point.
\end{thm}
\begin{proof} We show that the operators $A:K\To H, \,A(x)=g(x)-f(x)$ and $a:K\To H,\,a(x)=g(x)$ satisfy the assumptions in the hypothesis of Theorem \ref{t3.2}.

Let $L\subseteq g(K)$ be a finite dimensional subset . It is obvious that if the sequence $\big(g(x_n)\big)\subseteq L$ converges to $g(x)\in g(K)$  and the sequence $\big(f(x_n)\big)\subseteq H$ converges to $f(x)\in H$ in the  weak topology of  $H$ then the sequence $\big(g(x_n)-f(x_n)\big)\subseteq H$ converges to $g(x)-f(x)\in H$ in the weak topology of $H$.

From  $\|g(x)-g(y)\|^2\ge \la f(x)-f(y),g(x)-g(y)\ra $ for all $x,y\in K$  we obtain  $\la g(x)-g(y),g(x)-g(y)\ra \ge \la f(x)-f(y),g(x)-g(y)\ra $ for all $x,y\in K,$ or equivalently  $$\la (g-f)(x)-(g-f)(y),g(x)-g(y)\ra \ge 0$$ for all $x,y\in K,$ which shows that $g-f$ is  monotone relative to $g$.

Hence, according to Theorem \ref{t3.2}, $VI_S(g-f,g,K)$ admits solutions.

Let  $x\in K$ be a solution of $VI_S(g-f,g,K)$. Then we have  $\la g(x)-f(x),g(y)-g(x)\ra \ge 0$ for all $y\in K.$ Since $f(K)\subseteq g(K)$ choose $y\in K$ such that $g(y)=f(x).$ Then $\la g(x)-f(x),f(x)-g(x)\ra \ge 0$, or equivalently $-\|f(x)-g(x)\|^2\ge 0$, which leads to $f(x)=g(x).$
\end{proof}

As an immediate consequence we obtain the following fixed point result.

\begin{cor}\label{c4.5} Let $K \subseteq H$ be a weakly compact and convex set, let $f : K \To H$ be a given mapping such that $f(K) \subseteq K.$ Assume that $f$ is continuous on finite dimensional subspaces and  $\|x-y\|^2\ge \la f(x)-f(y),x-y\ra $ for all $x,y\in K.$ Then $f$ has a fixed point.
\end{cor}
\begin{proof} The conclusion follows from Theorem \ref{t4.2} by taking $g\equiv \id_K.$
\end{proof}

\begin{rem}\label{r4.1}\rm Note that if $f$ is $g-$nonexpansive, that is $\|f(x)-f(y)\|\le\|g(x)-g(y)\|$ for all $x,y\in K$, then the condition $\|g(x)-g(y)\|^2\ge \la f(x)-f(y),g(x)-g(y)\ra $ for all $x,y\in K$ in the hypothesis of Theorem \ref{t4.2} is satisfied, since  $\mbox{ for all }x,y\in K$ we have
$$\|g(x)-g(y)\|^2\ge \|f(x)-f(y)\|\|g(x)-g(y)\|\ge \la f(x)-f(y),g(x)-g(y)\ra.$$
 It particular, if $f$ is nonexpansive,   then the condition  $\|x-y\|^2\ge \la f(x)-f(y),x-y\ra $ for all $x,y\in K$ in Corollary \ref{c4.5} is satisfied, since in this case we have
$$\|x-y\|^2\ge \|f(x)-f(y)\|\|x-y\|\ge \la f(x)-f(y),x-y\ra \mbox{ for all }x,y\in K.$$
\end{rem}

\textbf{Acknowledgements.} The authors would like to express sincere thanks  to anonymous referees for their helpful comments and suggestions which led to improvement of the originally submitted version of this work.

\end{document}